\documentclass[11pt]{article}
\usepackage{enumitem,amsmath,amsfonts,amsthm,amssymb,fullpage}
\usepackage[hidelinks]{hyperref}
\usepackage[capitalize]{cleveref}
\usepackage[numbers,sort]{natbib}

\usepackage[mathscr]{euscript}

\newtheorem{theorem}{Theorem}
\newtheorem{lemma}[theorem]{Lemma}
\newtheorem{proposition}[theorem]{Proposition}
\newtheorem{corollary}[theorem]{Corollary}
\newtheorem{conjecture}[theorem]{Conjecture}
\theoremstyle{definition}
\newtheorem{definition}[theorem]{Definition}
\newtheorem{example}[theorem]{Example}

\crefname{equation}{equation}{equations}
\crefname{lemma}{Lemma}{Lemmas}
\crefname{proposition}{Proposition}{Propositions}
\crefname{claim}{Claim}{Claims}
\crefname{theorem}{Theorem}{Theorems}
\crefname{conjecture}{Conjecture}{Conjectures}
\crefname{figure}{Figure}{Figures}

\AddToHook{env/lemma/begin}{\crefalias{theorem}{lemma}}
\AddToHook{env/proposition/begin}{\crefalias{theorem}{proposition}}
\AddToHook{env/corollary/begin}{\crefalias{theorem}{corollary}}
\AddToHook{env/conjecture/begin}{\crefalias{theorem}{conjecture}}
\AddToHook{env/claim/begin}{\crefalias{theorem}{claim}}
\AddToHook{env/question/begin}{\crefalias{theorem}{question}} 
\AddToHook{env/definition/begin}{\crefalias{theorem}{definition}} 
\AddToHook{env/remark/begin}{\crefalias{theorem}{remark}} 
\AddToHook{env/example/begin}{\crefalias{theorem}{example}}

\DeclareMathOperator{\End}{End}
\DeclareMathOperator{\Gal}{Gal}
\DeclareMathOperator{\mdim}{maxdim}
\DeclareMathOperator{\rk}{rank}

\newcommand\ol[1]{\overline{#1}}
\newcommand\ab[1]{\lvert #1 \rvert}
\newcommand\smat[1]{\left(\begin{smallmatrix}#1\end{smallmatrix}\right)}

\let\leq\leqslant
\let\geq\geqslant

\newcommand{\F}{\mathbb{F}}
\newcommand{\C}{\mathbb{C}}
\newcommand{\Z}{\mathbb{Z}}
\newcommand{\T}{\mathscr{T}}
\newcommand{\X}{\mathscr{X}}
\newcommand{\Y}{\mathscr{Y}}
\newcommand{\M}{\mathscr{M}}
\newcommand{\R}{\mathbb{R}}
\newcommand{\Q}{\mathbb{Q}}
\newcommand\acf{\mathrm{ACF}}

\title{Duality for matrix space questions}
\author{Yuval Wigderson\thanks{IST Austria, 3400 Klosterneuburg, Austria. Email: \url{yuval.wigderson@ista.ac.at}. Research supported by Dr.\ Max R\"ossler, the Walter Haefner Foundation, and the ETH Z\"urich Foundation.}}
\date{}

\begin{document}

\maketitle
\begin{abstract}
	We present a new proof of a classical theorem of Dieudonn\'e: if a linear space of $n\times n$ matrices consists entirely of singular matrices, then its dimension is at most $n^2-n$. Our proof is based on a surprising ``duality'' argument: we prove this universal upper bound by exhibiting a \emph{single} matrix space that serves as a \emph{lower bound} for a related problem. Interestingly, this approach only works for certain fields, but we use model-theoretic arguments to obtain the same result for all fields. We hope that this approach can be generalized to provide new duality-based proofs of other classical theorems on matrix spaces, and give some preliminary results in this direction.
\end{abstract}

\section{Introduction.}
Let $F$ be a field, and let $F^{n\times n}$ denote the space of $n\times n$ matrices over $F$. $F^{n\times n}$ is a vector space over $F$, and a subspace\footnote{Throughout, we use the notation $U \leq V$ to mean that $U$ is a subspace of the vector space $V$.} $W \leq F^{n\times n}$ is called a \emph{matrix space} or \emph{space of $n\times n$ matrices over $F$}. The study of matrix spaces has rich and surprising connections to many areas of mathematics, including algebraic geometry \cite{MR954659}, algebraic topology \cite{MR179183}, graph theory \cite{MR4651017}, group theory \cite{MR1501972}, invariant theory \cite{MR29360}, matroid theory \cite{MR1129080}, quantum information theory \cite{MR4302212}, and theoretical computer science \cite{MR229540}. In most of these connections, one studies matrix spaces $W \leq F^{n\times n}$ such that every matrix in $W$ satisfies some natural linear-algebraic property.

One of the most interesting and well-studied cases concerns \emph{singular matrix spaces}, namely a space $S \leq F^{n\times n}$ such that every matrix $M \in S$ is singular (non-invertible). The study of singular matrix spaces goes back at least 75 years \cite{MR29360}, and many of the connections mentioned above actually concern singular matrix spaces. As a striking example of the complexity of this simple definition, let us remark that it is unknown how to efficiently determine if a given matrix space is singular; moreover, finding an efficient deterministic algorithm for this problem would come close\footnote{More precisely, it would prove much stronger circuit lower bounds than anything currently known.} to proving that $\mathrm P \neq \mathrm{NP}$ \cite{MR2105971,MR564634} (see also the survey \cite[Theorems 1.2 and 4.5]{MR2756166}).

At the other extreme, we might ask to study spaces of matrices all of whose elements are invertible. A moment's thought reveals that this makes no sense, since the zero matrix lies in every matrix space, and is certainly singular. However, if we exclude this trivial obstruction, we come to the definition of \emph{totally non-singular} matrix spaces, namely a space $T \leq F^{n\times n}$ such that every non-zero matrix $0 \neq M \in T$ is invertible. The structure of totally non-singular matrix spaces turns out to be quite interesting, and yields unexpected connections to other areas of mathematics; for example, determining the maximum dimension of an $n\times n$ totally non-singular space over $F=\R$ is equivalent \cite{MR179183} to determining the maximum number of non-vanishing vector fields on the sphere $S^{n-1}$, which was famously resolved by Adams using $K$-theoretic tools \cite{MR139178}. This question, in turn, is closely connected to the representation theory of Clifford algebras.

Having made these definitions, one is immediately led to the following trivial observation: if $S\leq F^{n\times n}$ is singular and $T \leq F^{n\times n}$ is totally non-singular, then
\begin{equation}\label{eq:S cap T}
	S \cap T = \{0\}.
\end{equation}
Indeed, any non-zero matrix in $S \cap T$ would need to simultaneously be invertible and non-invertible, hence cannot exist. Using the fact that $S,T \leq F^{n\times n}$, we conclude from \eqref{eq:S cap T} that
\begin{equation}\label{eq:dim sum}
	\dim S + \dim T \leq n^2,
\end{equation}
since $\dim {(F^{n\times n})}=n^2$. The goal of this paper is to show that some classical results about matrix spaces, whose proofs involve non-trivial algebraic machinery, can be derived from the extremely simple observations \eqref{eq:S cap T} and \eqref{eq:dim sum}.

\begin{example}\label{ex:2x2}
Let us consider first the case $F=\R, n=2$. It is not hard to come up with a 2-dimensional singular space of matrices; for example, we can take
\[
	S^* = \left\{ \begin{pmatrix}
	x&y\\0&0
	\end{pmatrix} : x,y \in \R \right\},
\]
which is certainly a singular matrix space as every matrix in $S^*$ has a row of zeros. It is a bit harder, but not by much, to come up with a 2-dimensional totally non-singular space of matrices in $\R^{2\times 2}$; for example, we can define
\begin{equation}\label{eq:complex 2x2}
	T^* = \left\{ \begin{pmatrix}
		x&y\\-y&x
	\end{pmatrix} :x,y \in \R \right\}.
\end{equation}
The determinant of the matrix $\smat{x&y\\-y&x}$ is $x^2+y^2$, which is strictly positive unless $x=y=0$, hence we see that $T^*$ is indeed totally non-singular. However, just from the \emph{existence} of these two simple examples, we can derive that they are the largest possible examples. 
\begin{proposition}
	The maximum dimension of a singular matrix space in $\R^{2\times 2}$ is $2$. Similarly, the maximum dimension of a totally non-singular matrix space in $\R^{2\times 2}$ is $2$.
\end{proposition}
\begin{proof}
	Since $T^*$ is totally non-singular and satisfies $\dim T^*=2$, we find from \eqref{eq:dim sum} that $\dim S \leq 4-\dim T^* = 2$ for any singular matrix space $S \leq \R^{2\times 2}$. Similarly, for any totally non-singular space $T$, \eqref{eq:dim sum} again shows $\dim T \leq 4-\dim S^*=2$.
\end{proof}
Let us pause and reflect on how remarkable this is: we have proven an upper bound on the dimension of \emph{any} singular matrix space $S$, simply by exhibiting \emph{a single} totally non-singular matrix space $T^*$ (and vice versa). That is, we have proved a ``for all'' statement by proving a ``there exists'' statement, which is (usually) a much simpler task. 
\end{example}

We can encapsulate the preceding discussion in the following more general statement, which is again a simple corollary of \eqref{eq:dim sum}.
\begin{proposition}\label{prop:duality}
	Suppose that $S^* \leq F^{n\times n}$ is singular and $T^* \leq F^{n\times n}$ is totally non-singular, and suppose that $\dim S^*+\dim T^*=n^2$. Then
	\[
		\max_{\text{\rm singular }S \leq F^{n\times n}} \dim S = \dim S^* \quad \text{ and } \quad \max_{\text{\rm totally non-singular }T \leq F^{n\times n}} \dim T = \dim T^*.
	\]
\end{proposition}
\begin{proof}
	By \eqref{eq:dim sum}, we have $\dim S \leq n^2 - \dim T^* = \dim S^*$ for any singular $S \leq F^{n\times n}$. Similarly, $\dim T \leq n^2-\dim S^* = \dim T^*$ for any totally non-singular $T$. 
\end{proof}
Again, the point of this is that the {existence} of a \emph{single} pair $(S^*,T^*)$ as in the proposition statement implies a sharp upper bound on the dimension of \emph{any} singular or totally non-singular space. 

Before proceeding, let us remark on the similarity between \cref{prop:duality} and the notion of \emph{duality}, which is one of the most important concepts in mathematical optimization (and in particular in linear, semidefinite, and convex programming). In such areas, our goal is to solve an optimization problem of the form $\max_{x \in \X} f(x)$, where $f$ is some ``nice'' (e.g.\ concave) function, and $\X$ is some ``nice'' (e.g.\ convex) domain. Under certain hypotheses, the theory of duality yields a \emph{dual} optimization problem, of the form $\min_{y \in \Y} g(y)$, for some other nice function $g$ and nice domain $\Y$, with the property that $f(x) \leq g(y)$ for every $x\in \X, y \in \Y$. In particular, this implies that the optima of the two optimization problems satisfy
\begin{equation}\label{eq:weak duality}
	\max_{x \in \X} f(x) \leq \min_{y \in \Y}g(y).
\end{equation}
As a consequence, if we can find a single pair $(x^*,y^*) \in \X \times \Y$ satisfying $f(x^*)=g(y^*)=m$, then we conclude that
\[
	\max_{x \in \X} f(x) \geq f(x^*) = m = g(y^*) \geq \min_{y \in \Y} g(y).
\]
Combined with \eqref{eq:weak duality}, we see that all these inequalities must actually be equalities, hence we
learn that this $m$ is the optimum value of both optimization problems: we must have $\max_{x \in \X} f(x)=f(x^*)=m=g(y^*)=\min_{y \in \Y}g(y)$. In other words, proving the ``for all'' statement $f(x) \leq m$ for all $x \in \X$ boils down to proving the existence of the single pair $(x^*,y^*)$. Moreover, in the setting of convex optimization, we actually have a much stronger fact:  such a pair $(x^*,y^*)$ is guaranteed to exist (under mild hypotheses). That is, such an optimization problem can \emph{always} be solved by exhibiting a single pair of solutions $(x^*,y^*)$.

At this level of abstraction, we see that \cref{prop:duality} fits into exactly the same framework. Namely, let us define $\X$ to be the set of singular matrix spaces $S\leq F^{n\times n}$, and $\Y$ to be the set of totally non-singular matrix spaces $T \leq F^{n\times n}$. If we further set $f(S) =\dim S$ and $g(T)=n^2-\dim T$, then \eqref{eq:dim sum} precisely states that $f(S) \leq g(T)$ for all $S \in \X, T \in \Y$. Additionally, \cref{prop:duality} says nothing more than if we can find some $S^* \in \X, T^* \in \Y$ with $f(S^*)=g(T^*)$, then we automatically solve both optimization problems. 

With that in mind, let us state the main theorem that we will prove using this duality approach.
\begin{theorem}\label{thm:dieudonne}
	Let $F$ be a field and $n$ a positive integer.
	The maximum dimension of a singular space $S \leq F^{n\times n}$ is $n^2-n$. 
\end{theorem}
\cref{thm:dieudonne} was first proved by Dieudonn\'e \cite{MR29360} under the assumption that $\ab F >n$, an assumption that was finally lifted by Meshulam \cite{MR790482}, who found a proof that works for all fields. Both proofs use somewhat involved algebraic techniques and properties of high-degree polynomials; our main goal in this paper is to demonstrate that \cref{thm:dieudonne} can be proved much more simply, as a consequence of the duality inequality \eqref{eq:dim sum}. Let us now see how this works.

First, the existence of a singular space $S^* \leq F^{n\times n}$ of dimension $\dim S^*=n^2-n$ is not too hard; for example, we can take $S^*$ to consist of all matrices whose last row is $0$, as we did in \cref{ex:2x2}. So, following the framework discussed above, all that remains to do is to find a totally non-singular space $T^* \leq F^{n \times n}$ of dimension $\dim T^*=n$, which will then imply \cref{thm:dieudonne} by \cref{prop:duality}.

Unfortunately, this last step is not always possible. For example, one can show that if $F$ is algebraically closed (e.g.\ if $F=\C$), then any totally non-singular space in $F^{n\times n}$ has dimension at most $1$. In order to circumvent this (serious) issue, we actually prove \cref{thm:dieudonne} by proving the following two lemmas.
\begin{lemma}\label{lem:algebraically open}
	If $F$ is a finite field, then there is a totally non-singular $T^* \leq F^{n\times n}$ with $\dim T^*=n$. 

	In particular, \cref{thm:dieudonne} holds whenever $F$ is a finite field.
\end{lemma}
\begin{lemma}\label{lem:reduction}
	If \cref{thm:dieudonne} holds whenever $F$ is a finite field, then it holds for all fields.
\end{lemma}
The proof of \cref{lem:algebraically open} is fairly short and straightforward, and actually works for a larger class of fields which we term \emph{algebraically open fields}, which includes some infinite fields such as $\Q$. The construction of $T^*$ for finite fields is fairly well-known, and the fact that it can be used to provide a proof of \cref{thm:dieudonne} for finite fields was also observed in \cite[Corollary 4]{MR1324051}.

Thus, in some sense, our main contribution is the proof of \cref{lem:reduction}, which is somewhat involved: we first use an elementary reduction to prove that \cref{thm:dieudonne} holds for the algebraically closed fields $\ol{\F_p}$ for all primes $p$, then use model-theoretic techniques (related to the proof of the Ax--Grothendieck theorem \cite{MR229613,MR217086}) to deduce that \cref{thm:dieudonne} holds for all algebraically closed fields, and finally use another elementary reduction to deduce the statement for all fields.

The reader might justly complain that they were promised a \emph{simple} proof of \cref{thm:dieudonne}, and this multi-step process looks rather convoluted. This is a fair criticism, but we believe that our proof of \cref{thm:dieudonne} has some advantages over the earlier proofs of Dieudonn\'e and Meshulam. First, for many fields that we are interested in (e.g.\ the rationals and all finite fields), only \cref{lem:algebraically open} is necessary, and hence for such fields the proof really is quite simple and follows the duality framework laid out above. Note, in particular, that small finite fields are precisely the fields for which Dieudonn\'e's original proof fails, and we obtain a very simple proof for such fields.

Second, while the proof of \cref{lem:reduction} is somewhat involved, it is also extremely general, and uses almost nothing about the fact that we are working with singular spaces. In fact, we prove a much more general result (see \cref{thm:general reduction} below), which states that for a wide variety of natural properties of matrix spaces, to upper-bound the maximum dimension of a space whose matrices all have this property, it suffices to solve this problem over all finite fields. 

In particular, we are hopeful that the duality technique can be used to prove other extremal results in matrix spaces. As a proof of concept, we prove the following generalization of \cref{thm:dieudonne} (which is precisely the case $m=n, r=n-1$ of the following result).
\begin{theorem}\label{thm:flanders}
	Let $F$ be a field and let $r \leq m \leq n$ be positive integers. The maximum dimension of a space $R \leq F^{m \times n}$ all of whose elements have rank at most $r$ is $rn$. 
\end{theorem}
The lower bound is again easy: the space $R^*$ consisting of those matrices whose last $m-r$ rows are zero consists only of matrices of rank at most $r$, and has dimension $rn$. The upper bound in \cref{thm:flanders} was first proved by Flanders \cite{MR136618} under the assumption that $\ab F>r$, and the general case was again proved by Meshulam \cite{MR790482}. Again, we provide a new duality-based proof of \cref{thm:flanders}: we first show that for all finite fields, we can construct a space $Q^* \leq F^{m\times n}$ with $\dim Q^*=(m-r)n$ such that all non-zero matrices in $Q^*$ have rank strictly larger than $r$. We then have
\[
	\dim R^* + \dim Q^* = rn + (m-r)n = mn = \dim F^{m \times n},
\]
which implies \cref{thm:flanders} for all finite fields, via the same argument as in \cref{prop:duality}. Finally, we then use (the general form of) \cref{lem:reduction} to deduce that the same result holds for all fields. 

There are many other examples of extremal results on matrix spaces, and we hope that some of them can be proved along the same lines. A notable example is a theorem of Gerstenhaber and Serezhkin \cite{MR96678,MR828449}, which states that if $N \leq F^{n\times n}$ consists entirely of nilpotent matrices, then $\dim N \leq \binom n2$. It would again be very interesting to find a proof of this theorem using duality, i.e.\ by exhibiting a space of dimension $n^2-\binom n2$ none of whose non-zero elements is nilpotent. This is straightforward over $\R$, as the space of symmetric matrices does the job, but it is not clear to us how to do it (or whether it is even true) over finite fields.
\begin{conjecture}
For every finite field $F$ and every integer $n \geq 1$, there is a matrix space $W \leq F^{n\times n}$ with $\dim W = n^2-\binom n2$ such that no non-zero matrix in $W$ is nilpotent. 
\end{conjecture}

If one could find such a construction, our general form of \cref{lem:reduction} would again yield the Gerstenhaber--Serezhkin theorem over all fields. A more ambitious goal would be to use this approach to prove a theorem of Atkinson \cite{MR571726}, which simultaneously generalizes \cref{thm:dieudonne} and the Gerstenhaber--Serezhkin theorem; Atkinson's proof is again restricted only to large fields, and to the best of our knowledge, there is currently no known proof of Atkinson's theorem that works over all fields.

\section{Algebraically open fields.}
In this section, we prove \cref{lem:algebraically open}. As indicated above, we actually prove it for a more general class of fields than just finite fields. 

Before making a formal definition, let us think back to \eqref{eq:complex 2x2}. There, we considered the space of all matrices of the form $\smat{x&y\\-y&x}$ over $\R$, and observed that this space is totally non-singular. This matrix space is in fact an algebra (the product of two such matrices is another matrix of the same form). Moreover, a bit of thought will reveal that this algebra is isomorphic to $\C$: the matrix $\smat{x&y\\-y&x}$ corresponds to the complex number $x+iy$. The fact that this matrix space is totally non-singular boils down to the fact that $\C$ is a field: every non-zero element in $\C$ is invertible. 

With this understanding of the construction, it is fairly straightforward to generalize it as follows. We recall that a \emph{field extension} $F \subset K$ is just an inclusion of one field in another; in any such situation, the field $K$ is a vector space over $F$, and the \emph{degree} of the extension is the dimension of this vector space.
\begin{proposition}\label{prop:deg n extension}
	Suppose that $F \subset K$ is a field extension of degree $n$. Then there is a totally non-singular matrix space $T^* \leq F^{n\times n}$ with $\dim T^*=n$. 
\end{proposition}
\begin{proof}
	For every $\alpha \in K$, consider the multiplication map $m_\alpha:K \to K$ defined by $m_\alpha(x)=\alpha x$. This is an $F$-linear map, hence $m_\alpha \in \End_F(K)\cong F^{n\times n}$. Moreover, we clearly have that $m_{\alpha+\beta}=m_\alpha +m_\beta$ and that $m_{\gamma \alpha}=\gamma m_\alpha$ for $\gamma \in F$, hence the set of all $m_\alpha$ defines a vector subspace of $F^{n\times n}$, and we define $T^*=\{m_\alpha : \alpha \in K\} \leq F^{n\times n}$ to be this subspace.

	Moreover, as each non-zero $\alpha$ is invertible in $K$, we see that $m_\alpha$ is an invertible linear map. This implies that every non-zero element of $T^*$ is invertible, hence $T^*$ is totally non-singular. Moreover, it implies the $F$-linear map $\alpha \mapsto m_\alpha$ is injective, as no non-zero $\alpha$ can be in its kernel. As the domain of this map is $K$, which is $n$-dimensional, we conclude that $\dim T^*=n$. 
\end{proof}
This proposition is exactly what we wanted: it gives a general setting under which we can find a totally non-singular space of dimension $n$ in $F^{n\times n}$. All we need to apply this is access to a field extension $K \supset F$ of degree $n$. This motivates the following definition\footnote{We have not been able to find this notion in the literature, so do not think it has a standard name. The name ``algebraically open'' is meant to convey that such a field lies at the opposite extreme from an algebraically closed field, which has no non-trivial finite-degree extensions whatsoever.}. 
\begin{definition}
	A field $F$ is called \emph{algebraically open} if, for any positive integer $n$, there is a field extension $K \supset F$ of degree $n$. 
\end{definition}
Thanks to \cref{prop:deg n extension}, we see that if $F$ is algebraically open, then there is an $n$-dimensional totally non-singular space in $F^{n\times n}$ for all $n$. Thus, in order to complete the proof of \cref{lem:algebraically open}, it suffices to observe the following.
\begin{lemma}\label{lem:finite and number are open}
	All finite fields are algebraically open.
\end{lemma}
\begin{proof}
	This is an immediate and well-known corollary of the classification of finite fields: for every prime power $q$ and every integer $n$, there is a field $\F_{q^n}$ which is a degree-$n$ extension of the field $\F_q$, implying that $\F_q$ is algebraically open. As every finite field is of the form $\F_q$ for some prime power $q$, we have the desired conclusion. 
\end{proof}
Combining \cref{prop:deg n extension,lem:finite and number are open}, we obtain \cref{lem:algebraically open}.
Given that we only need \cref{lem:algebraically open} for finite fields, there is arguably not much advantage to defining algebraically open fields and proving \cref{prop:deg n extension} in this greater generality. However, the reason we did so is that many other fields besides finite fields are algebraically open. For example, it is a well-known consequence of properties of cyclotomic fields that $\Q$ is algebraically open (see e.g.\ \cite[Corollary 14.28]{MR2286236}). With more effort, one can prove that in fact, every number field is algebraically open\footnote{This also has an elementary proof using cyclotomic fields, or can alternatively be obtained as a very special case of the Grunwald--Wang theorem, a deep result in class field theory (see e.g.\ \cite[Theorem X.6]{MR2467155}).}. Thus, \cref{prop:deg n extension} gives an extremely simple proof of \cref{thm:dieudonne} for all such fields.

The other reason we work in the greater generality of algebraically open fields is that, more generally, the existence of totally non-singular spaces over $F$ is closely related to the set of field extensions of $F$. We have already mentioned that algebraically closed fields do not admit any non-trivial totally non-singular spaces. Over $\R$, the largest $n\times n$ totally non-singular spaces arise from the action of Clifford algebras, and this construction can be viewed as a generalization of \cref{prop:deg n extension} where we allow $K$ to be a non-commutative ring. For a more precise statement of this connection, see \cite[Section 4]{MR2645099}.

\section{The reduction to general fields.}
In this section, we state and prove a generalization of \cref{lem:reduction}, which implies that for a wide array of extremal problems on matrix spaces, a bound that holds over finite fields immediately implies that the same bound holds over all fields. We begin by defining the class of questions that this reduction theorem applies to. Roughly speaking, this is the class of matrix properties that can be defined by polynomials with integer coefficients.

Before giving the formal definition, we recall that for every field $F$, there is a unique ring homomorphism $\Z \to F$. This map naturally extends to a map of polynomial rings $\Z[x_1,\dots,x_\ell] \to F[x_1,\dots,x_\ell]$, and this allows us to view any integer polynomial as a polynomial defined over $F$. We will frequently abuse notation and use the same symbol to refer both to the polynomial with coefficients in $\Z$ and the polynomial with coefficients in $F$.
\begin{definition}
	Let $p_1,\dots,p_k \in \Z[x_{11},\dots,x_{mn}]$ be polynomials with integer coefficients in the variables $\{x_{ij}:1\leq i\leq m, 1 \leq j \leq n\}$. Let $P = \{p_1,\dots,p_k\}$, and let $Z(P) \subseteq F^{m\times n}$ be the set of $m \times n$ matrices $M$ satisfying $p_1(M) = \dots = p_k(M)=0$; here, we view $p_1,\dots,p_k$ as polynomials over $F$, and we plug in the $mn$ entries of $M$ as the variables of the polynomials. If $M \in Z(P)$, we say that $M$ \emph{satisfies} $P$. 

	Finally, we say that a matrix space $W \leq F^{m\times n}$ is a \emph{$P$-space} if $W \subseteq Z(P)$. That is, a $P$-space is a matrix space all of whose elements satisfy $P$.
\end{definition}
\begin{example}\label{ex:P-spaces}
	The determinant of an $n\times n$ matrix is an integer polynomial in the entries of the matrix. Hence, if we let $P_{\mathrm{sing}} = \{\det\}$, then a matrix satisfies $P_{\mathrm{sing}}$ if and only if its determinant is $0$, i.e.\ if and only if it is singular. Therefore, a $P_{\mathrm{sing}}$-space is the same as a singular space. 

	Similarly, the property of having rank at most $r$ can be encoded by polynomials. Namely, let $P_{\leq r}$ denote the set of all determinants of all $(r+1)\times (r+1)$ submatrices of the $m \times n$ matrix $\smat{x_{11}&\dotsb&x_{1n}\\\vdots &\ddots&\vdots\\x_{m1}&\dotsb&x_{mn}}$. Then a matrix satisfies $P_{\leq r}$ if and only if each of its $(r+1)\times (r+1)$ submatrices are singular, which happens if and only if its rank is at most $r$. Hence, a $P_{\leq r}$-space is the same as a matrix space all of whose elements have rank at most $r$. 

	Another example is the property of nilpotency, which can be encoded by polynomials in at least two different ways. First, an $n\times n$ matrix $M$ is nilpotent if and only if $M^{n}=0$. The $n^2$ entries of $M^{n}$ are each integer polynomials of the entries in $M$, so we can encode nilpotency by insisting that these $n^2$ polynomials are all zero. Second, a matrix is nilpotent if and only if its eigenvalues are all zero, which happens if and only if all coefficients of its characteristic polynomial, apart from the leading coefficient, are zero. The coefficients of the characteristic polynomial are themselves polynomials in the entries of the matrix, so we could instead demand that these $n$ polynomials in the entries are all zero.\footnote{As pointed out by an anonymous referee, the same idea can be used to obtain a different encoding of the property of rank at most $r$, at least for square matrices: an $n\times n$ matrix has rank at most $r$ if and only if it has at least $n-r$ zero eigenvalues, which happens if and only if the $n-r$ lowest coefficients of its characteristic polynomial are zero.}

	Finally, totally non-singular spaces are a non-example: there is no collection of polynomials $P$ such that a matrix space is totally non-singular if and only if it is a $P$-space. 
\end{example}
We now introduce one final piece of notation: for a set $P\subseteq \Z[x_{11},\dots,x_{mn}]$ of integer polynomials and a field $F$, let us denote by $\mdim_F(P)$ the maximum dimension of a $P$-space $W \leq F^{m\times n}$. 
With these definitions in hand, we can now state our general reduction theorem.
\begin{theorem}\label{thm:general reduction}
	Let $P \subseteq \Z[x_{11},\dots,x_{mn}]$ be a finite set of integer polynomials. Suppose that
	\begin{equation}\label{eq:mdim}
		\mdim_F(P)\leq d
	\end{equation}
	for every finite field $F$. Then the same holds for any field.
\end{theorem}
In particular, this theorem immediately implies \cref{lem:reduction}, which corresponds to the case $m=n, d=n^2-n$, and $P = P_{\mathrm{sing}} = \{\det\}$. 

The proof of \cref{thm:general reduction} breaks into three steps, which are accomplished in each of the following three subsections. For the remainder of this section, we treat $P,m,n,d$ as fixed. 

\subsection{From finite fields to \texorpdfstring{$\ol {\F_p}$}{\string\\bar F\_p}.}
We begin by showing that if \eqref{eq:mdim} holds for all finite fields, then it holds for their algebraic closures $\ol{\F_p}$, for all primes $p$.

Thus, let us fix some $P$-space $W \leq {\ol {\F_p}}^{m\times n}$; our goal is to show that $\dim W \leq d$. We fix a basis $M_1,\dots,M_t$ for $W$. Recall that $\ol{\F_p}$ is the union of the finite fields $\F_{p^k}$ for all $k$; this means that every element of $\ol{\F_p}$ is an element of $\F_{p^k}$ for some $k$. In particular, we can find some integer $k$ such that each of the $tmn$ entries of $M_1,\dots,M_t$ all lie in the field $F= \F_{p^k}$. 

Let $W_F \leq F^{m\times n}$ be the matrix space over $F$ spanned by these matrices. As $M_1,\dots,M_t$ are linearly independent over $\ol{\F_p}$, they are also linearly independent over the subfield $F$, hence $\dim W_F=\dim W$. Moreover, every matrix in $W_F$ is also a matrix in $W$: it is an $F$-linear combination of $M_1,\dots,M_t$, hence it is also a $\ol{\F_p}$-linear combination of $M_1,\dots,M_t$, hence it lies in $W$. In particular, every matrix in $W_F$ satisfies $P$, since $W$ is a $P$-space. As a consequence,
\[
	\dim W = \dim W_F \leq \mdim_F(P) \leq d,
\]
which is what we wanted to prove.

\subsection{From \texorpdfstring{$\ol {\F_p}$}{\string\\bar F\_p} to all algebraically closed fields via model theory.}
We now turn to the crux of the argument, where we use model-theoretic tools to convert our knowledge of the bound \eqref{eq:mdim} from the special algebraically closed fields $\ol{\F_p}$ to all algebraically closed fields. Such arguments are now fairly commonplace, and to our knowledge they originate with the classical proofs of the Ax--Grothendieck theorem \cite{MR229613,MR217086}. This is a non-trivial statement about polynomial maps $\C^n \to \C^n$, but the analogous statement for $\ol{\F_p}^n$ is rather simple; the observation of Ax and Grothendieck is that one can deduce the statement for $\C$ from the statement for $\ol{\F_p}$. 

To state the main model-theoretic result we need, we recall that a \emph{first-order sentence in the language of rings} is any statement that can be formed with the quantifiers $\forall$ and $\exists$, the logical operations $\neg,\wedge,\vee$, variables for ring elements, and the symbols $0,1,+,-,\times$ corresponding to the additive and multiplicative identities, and to the basic operations in a ring. 

A given first-order sentence $\phi$ can be true in some rings but not in others. For example, the sentence $1+1=0$ is true in rings of characteristic $2$ but in no others. The sentence $\forall x:(x=0)\vee(\exists y:x\times y=1)$ states that every non-zero element has a multiplicative inverse, and thus is true in all fields but false in any ring that is not a field. The sentence $\exists x: x\times x+1=0$ is true in any algebraically closed field, but not true in fields such as $\R$ in which $-1$ is not a perfect square. 

Most importantly for our purposes, the statement ``$\mdim_F(P)\leq d$'' can be encoded as a first-order sentence in the language of rings, such that this sentence is true in $F$ if and only if we actually have $\mdim_F(P)\leq d$. Indeed, the statement $\mdim_F(P)\leq d$ holds if and only if for all linearly independent $M_1,\dots,M_{d+1} \in F^{m\times n}$, there is some matrix in their span which does not satisfy $P$. The fact that $M_1,\dots,M_{d+1}$ are linearly independent can be expressed as a first-order sentence (for all $\lambda_1,\dots,\lambda_{d+1}$, either $\lambda_1=\dots=\lambda_{d+1}=0$, or the linear combination $\sum \lambda_i M_i$ is non-zero), and the fact that there exists a matrix in the span not satisfying $P$ is also a first-order statement (there exist $\lambda_1,\dots,\lambda_{d+1}$ such that for $M=\sum\lambda_i M_i$, either $p_1(M)\neq 0$, or $p_2(M)\neq 0$, or\dots). We remark that this is the only place in the proof where we use that the polynomials in $P$ have integer coefficients: this fact means that every polynomial in $P$ can be expressed using variables and the symbols $0,1,+,-,\times$.

In particular, the next step of our reduction follows immediately from the following theorem. 
\begin{theorem}\label{thm:model theory}
	Let $\phi$ be a first-order sentence in the language of rings. Then the following are equivalent.
	\begin{enumerate}[label=(\roman*)]
		\item $\phi$ holds in the fields $\ol{\F_p}$ for all primes $p$.\label{it:Fp bar}
		\item $\phi$ holds in all algebraically closed fields. \label{it:all acf}
	\end{enumerate}
\end{theorem}
In particular, given that we now know that $\mdim_F(P)\leq d$ holds for $F=\ol{\F_p}$ for all $p$, we conclude that it holds for all algebraically closed fields. 

As \cref{thm:model theory} may (and should!) appear completely magical to those unfamiliar with the model-theoretic background, we include a proof sketch of \cref{thm:model theory} in \cref{sec:model theory}, and a full proof can be found in \cite[Corollaries 2.2.9 and 2.2.10]{MR1924282}. For further applications of this result and its variants, we recommend the survey \cite{MR2555994}. However, for the moment, we continue with our proof of \cref{thm:general reduction}.

\subsection{From algebraically closed fields to all fields.}
Before we begin the final step, we record the following simple lemma.
\begin{lemma}\label{lem:lin indep extension}
	Let $F,K$ be fields with $F \subseteq K$. Let $v_1,\dots,v_t \in F^k$ be linearly independent over $F$. Then $v_1,\dots,v_t$ are also linearly independent over $K$, when viewed as vectors in the vector space $K^k$. 
\end{lemma}
\begin{proof}
	Let $V$ be the $k\times t$ matrix whose columns are $v_1,\dots,v_t$. As $v_1,\dots,v_t$ are linearly independent over $F$, we see that $V$ has rank $t$, when viewed as a matrix over $F$. Therefore, there is some collection of $t$ rows of $V$ such that the corresponding $t\times t$ submatrix $V$ also has rank $t$ over $F$. In particular, this implies that $\det V'$ is some non-zero element of $F$.

	However, if we now turn to compute $\det V'$ over the field $K$, we will get the same answer, as we are working with the same matrix and the determinant is computed identically over all fields. Hence $\det V'$ is also a non-zero element of $K$, meaning that $V'$ is invertible over $K$. Thus $V'$ has rank $t$ over $K$, and hence so does $V$. As a consequence, the columns of $V$, which are the vectors $v_1,\dots,v_t$, are linearly independent over $K$.
\end{proof}

Let now $F$ be a field, and let $\ol F$ be its algebraic closure. Thanks to \cref{thm:model theory}, we know that $\mdim_{\ol F}(P)\leq d$. Our goal now is to prove that $\mdim_F(P)\leq d$ as well. 

As the assumption of \cref{thm:general reduction} is that \eqref{eq:mdim} holds for all finite fields, we may assume that $F$ is infinite. Suppose for contradiction that $\mdim_F(P)>d$, namely that there exists some $P$-space $W \in F^{m\times n}$ with $\dim W>d$. Fix a basis $M_1,\dots,M_t$ for $W$, and note that by \cref{lem:lin indep extension}, $M_1,\dots,M_t$ are linearly independent over $\ol F$. Therefore, if we let $\ol W \in \ol F^{m\times n}$ be the $\ol F$-span of $M_1,\dots,M_t$ we have that $\dim \ol W=\dim W > d \geq \mdim_{\ol F}(P)$. Therefore, $\ol W$ must not be a $P$-space, hence there exists some $M \in \ol W$ not satisfying $P$. In particular, there is some fixed polynomial $p \in P$ such that $p(M) \neq 0$.

Let us define the polynomial $q \in F[y_1,\dots,y_t]$ by
\[
	q(y_1,\dots,y_t) = p(y_1 M_1 + \dots +y_t M_t).
\]
Since the entries of $M_1,\dots,M_t$ are in $F$, we see that $q$ is indeed a polynomial with coefficients in $F$. Since $F \subseteq \ol F$, we can also view it as a polynomial over $\ol F$. 

Since $M\in \ol W$, there exist some $\lambda_1,\dots,\lambda_{t} \in \ol F$ such that $M= \lambda_1 M_1 + \dots + \lambda_t M_t$. And since $p(M) \neq 0$, we conclude that $q(\lambda_1,\dots,\lambda_t)\neq 0$. In particular, we see that $q$ is not the zero polynomial; if it were, we would have $q(\lambda_1,\dots, \lambda_t)=0$ for all $(\lambda_1,\dots,\lambda_t) \in \ol F^t$. 

To conclude, we will use the following simple and well-known lemma about polynomials.
\begin{lemma}\label{lem:schwartz zippel}
	If $F$ is an infinite field and $q \in F[y_1,\dots,y_t]$ is not the zero polynomial, then there exists some $(\mu_1,\dots,\mu_t) \in F^t$ such that $q(\mu_1,\dots,\mu_t)\neq 0$. 
\end{lemma}
\cref{lem:schwartz zippel} is essentially (a weak form of) the Schwartz--Zippel lemma, and we include a short proof for completeness momentarily. But before we do, let us conclude the proof of \cref{thm:general reduction}. Applying \cref{lem:schwartz zippel} to the polynomial $q$ defined above, we obtain $(\mu_1,\dots,\mu_t) \in F^t$ with $q(\mu_1,\dots,\mu_t)\neq 0$. Let $M' =\mu_1 M_1 + \dots + \mu_t M_t$, and note that $M' \in W$, as $M'$ is an $F$-linear combination of the matrices $M_1,\dots,M_t$. Moreover, since $q(\mu_1,\dots,\mu_t)\neq 0$, we conclude that $p(M') \neq 0$. But this contradicts our assumption that $W$ is a $P$-space, completing the proof. 

As promised, we now present the proof of \cref{lem:schwartz zippel}.
\begin{proof}[Proof of \cref{lem:schwartz zippel}]
	We proceed by induction on $t$. The base case $t=1$ is immediate, since a non-zero polynomial in one variable has at most as many roots as its degree, and in particular has only finitely many roots; since $F$ is infinite it must have some non-root in $F$. 

	For the inductive step, let us write
	\[
		q(y_1,\dots,y_t) = \sum_{i=0}^s y_t^i q_i(y_1,\dots,y_{t-1})
	\]
	for some polynomials $q_0,\dots,q_s \in F[y_1,\dots,y_{t-1}]$, where $s$ is the degree of $y_t$ in $q$. Note that we may assume that $s>0$, since otherwise $y_t$ does not appear at all in $q$, and we are already done by the inductive hypothesis. This in particular implies that $q_s$ is not the zero polynomial. By the inductive hypothesis, we can find some $\mu_1,\dots,\mu_{t-1}\in F$ with $q_s(\mu_1,\dots,\mu_{t-1})\neq 0$. But this implies that $q(\mu_1,\dots,\mu_{t-1},y_t)$, viewed as a one-variable polynomial, is not the zero polynomial. So by again applying the base case we can find some $\mu_t$ such that $q(\mu_1,\dots,\mu_t)\neq 0$, as claimed. 
\end{proof}

\section{Proof sketch of Theorem \ref{thm:model theory}.}\label{sec:model theory}
In this section, we give a brief sketch of the model-theoretic background needed to prove \cref{thm:model theory}. We stress that while this sketch conveys all of the key ideas, it is by no means a full proof, and we refer to a standard introductory text in model theory such as \cite{MR1924282} for more formal definitions and proofs. 

For our purposes, a \emph{theory} is simply a collection of axioms, each of which is a first-order sentence in the language of rings\footnote{Of course, there is no reason to restrict oneself to the language of rings, and indeed the real definition of a theory works with an arbitrary language. But our only goal is \cref{thm:model theory}, so we restrict ourselves to this less general setting. We also assume throughout that our rings are commutative and have a multiplicative identity. }. For example, the theory of rings consists of all of the usual ring axioms, e.g.\ the axiom ``$\forall x:x+0=x$'' that $0$ is the additive identity, the axiom ``$\forall x \forall y:x+y=y+x$'' that addition is commutative, and the axiom ``$\forall x \forall y \forall z:x\times(y+z)=(x\times y)+(x\times z)$'' that multiplication distributes over addition. If we add to these axioms the axiom ``$\forall x:(x=0)\vee(\exists y:x\times y=1)$'' that every non-zero element is invertible, we obtain the theory of fields. We can then add further axioms to get the theory of algebraically closed fields; this involves adding, for every degree $d$, an axiom stating that every monic polynomial of degree $d$ has a root.

A closely related concept is that of a \emph{model}. Loosely speaking, a model $\M$ of a theory $\T$ is an object in which all the axioms of $\T$ hold. Thus, for instance, $\Z$ is a model of the theory of rings, but is not a model of the theory of fields; this is saying nothing more but that $\Z$ is a ring and is not a field. 

Given a theory $\T$ and a sentence $\phi$, we say that \emph{$\T$ proves $\phi$} if $\phi$ is true in every model $\M$ of $\T$. For example, an elementary exercise in any first course on rings is that additive inverses are unique; this means that the theory of rings proves the sentence $\forall x \forall y \forall z:(y= z)\vee (x+y\neq 0) \vee(x+z\neq 0)$. 
The reader might complain that this is not what ``prove'' should mean. Instead, it should mean that one can perform a series of logical deductions to derive $\phi$ from the axioms of $\T$. Indeed, when one solves the elementary ring theory exercise above, one does not actually verify that this property holds by inspecting all rings one by one, but instead derives this property from the axioms of rings. Luckily, the famous \emph{completeness theorem} of G\"odel says that these two notions of proof are equivalent: one can logically deduce $\phi$ from the axioms in $\T$ if and only if $\phi$ is true in every model of $\T$. 

A theory 
is called \emph{complete} if it can determine the veracity of any statement, i.e.\ for every sentence $\phi$, it can either prove $\phi$ or $\neg \phi$. 
This is a very strong condition: it implies that for every first-order sentence $\phi$, either $\phi$ is true in all models or it is false in all models. Thus, for example, we can easily see that the theory of rings is incomplete: the fact that some rings are fields and some are not implies that the theory of rings cannot prove or disprove the axiom ``all non-zero elements have a multiplicative inverse'', and hence the theory of rings is incomplete. Similarly, the theory of fields is incomplete, since some fields are algebraically closed and others are not. And again, the theory of algebraically closed fields is incomplete, because it can neither prove nor disprove the sentence ``$1+1=0$''; this is because there are algebraically closed fields of characteristic $2$, and also algebraically closed fields of other characteristics.

Rather remarkably, however, this is essentially the only obstruction to completeness. More precisely, let $\acf$ denote the theory of algebraically closed fields. For every prime $p$, let $\chi_p$ denote the sentence $1+\dots+1=0$, where there are $p$ ones on the left-hand side. Then let $\acf_p = \acf \cup \{\chi_p\}$ denote the theory of algebraically closed fields of characteristic $p$, and let $\acf_0 = \acf \cup \{\neg \chi_2, \neg \chi_3,\neg \chi_5,\dots\}$ denote the theory\footnote{We are implicitly using here the fact that every algebraically closed field either has characteristic $0$ or characteristic $p$ for some prime $p$, hence if we add in the negations of all the axioms $\chi_p$, we get the theory of algebraically closed fields of characteristic $0$.} of algebraically closed fields of characteristic $0$. Then we have the following remarkable theorem.
\begin{theorem}\label{thm:acf complete}
	The theories $\acf_0$ and $\acf_p$, for every prime $p$, are complete. 
\end{theorem}
The proof of \cref{thm:acf complete} makes it seem a bit less remarkable. Indeed, unpacking what it means, it states that for every sentence $\phi$, either $\phi$ or $\neg \phi$ is true in every algebraically closed field of a given characteristic. But the reason for this is just that there are not so many different algebraically closed fields of a given characteristic: a standard result in field theory is that algebraically closed fields are uniquely determined by their characteristic and their transcendence degree (see e.g.\ \cite[Theorem 9.16]{MR4506896}). To prove \cref{thm:acf complete}, suppose for contradiction that $F,F'$ are two algebraically closed fields of the same characteristic, such that $\phi$ is true in $F$ but false in $F'$. By adding uncountably many ``dummy elements'' (see \cite[Proposition 2.2.2]{MR1924282} for details), we can moreover assume that $F$ and $F'$ have the same uncountable cardinality. But an uncountable field has transcendence degree equal to its cardinality, hence $F$ and $F'$ must be isomorphic, a contradiction. 

Given \cref{thm:acf complete}, it is quite easy to prove \cref{thm:model theory}.
\begin{proof}[Proof of \cref{thm:model theory}]
	Clearly, \ref{it:all acf} implies \ref{it:Fp bar}, so we only need to prove the converse. So suppose that $\phi$ is true in $\ol{\F_p}$ for all primes $p$. This means that $\phi$ is true in some model of $\acf_p$ (namely $\ol{\F_p}$), hence it must be true in all models of $\acf_p$, by the completeness of $\acf_p$. That is, $\phi$ is true in every algebraically closed field of positive characteristic.

	So all that remains is to handle the case of characteristic $0$. Suppose for contradiction that $\acf_0$ does not prove $\phi$. By the completeness of $\acf_0$, this means that $\acf_0$ proves $\neg \phi$. Recall that, by the completeness theorem, this means that one can deduce $\neg\phi$ from the axioms of $\acf_0$ via some finite sequence of logical deductions. In this finite sequence of deductions, we must only use finitely many axioms of $\acf_0$, and in particular only use finitely many of the axioms $\neg \chi_p$. But if $p_0$ is some other prime not in this finite list, then the same proof is also valid in the theory $\acf_{p_0}$. Unpacking what all this means, we have found that if $\acf_0$ proves $\neg \phi$, then $\acf_{p_0}$ proves $\neg \phi$ as well, for all but finitely many primes $p_0$. But this is impossible, as we just saw that $\acf_{p_0}$ proves $\phi$ for all $p_0$. 
\end{proof}

\section{Proof of Theorem \ref{thm:flanders}.}
In this section, we use the same duality machinery we have developed to prove \cref{thm:flanders}. Recall the statement: for all integers $r \leq m \leq n$, if every matrix in $R \leq F^{m\times n}$ has rank at most $r$, then $\dim R \leq rn$. As the property of having rank at most $r$ can be described by integer polynomials (\cref{ex:P-spaces}), \cref{thm:general reduction} implies that it suffices to prove this statement for finite fields. And over finite fields, we have the following elegant result of Roth \cite{MR1093747}.
\begin{theorem}\label{thm:roth}
	If $F$ is a finite field and $s \leq n$ are non-negative integers, then there is a matrix space $Q \leq F^{n\times n}$ of dimension $\dim Q=(n-s)n$, all of whose non-zero elements have rank greater than $s$. 
\end{theorem}
Note that the case $s=n-1$ is precisely \cref{lem:algebraically open}. Although \cref{thm:roth} is stated only for $n\times n$ matrices, it gives as a simple corollary an analogous result for rectangular matrices.
\begin{corollary}\label{cor:roth rectangular}
	If $F$ is a finite field and $r \leq m \leq n$ are integers, there is a matrix space $Q^* \leq F^{m\times n}$ of dimension $\dim Q^*= (m-r)n$, all of whose non-zero elements have rank greater than $r$.
\end{corollary}
\begin{proof}
	Let $Q\leq F^{n\times n}$ be the space given by \cref{thm:roth}, applied with parameter $s=n-m+r$, so that $\dim Q = (n-s)n=(m-r)n$. Let $Q^*$ be obtained from $Q$ by deleting the last $n-m$ rows of each matrix, to obtain a matrix space in $F^{m\times n}$. More formally, letting $\pi:F^{n\times n} \to F^{m\times n}$ be the projection sending each matrix to the submatrix given by the first $m$ rows, we define $Q^* = \pi(Q)$. By the rank-nullity theorem, we thus have
	\[
		\dim Q^* = \dim Q - \dim(Q \cap \ker \pi).
	\]
	We claim that $Q \cap \ker \pi = \{0\}$. Indeed, if $M$ is a non-zero matrix in $Q \cap \ker \pi$, then $M$ must have rank greater than $s\geq n-m$ (by the definition of $Q$), and must also be supported entirely in the last $n-m$ rows (by the definition of $\pi$), and these two properties are incompatible. Therefore $\dim(Q\cap \ker\pi)=0$, and we find that $\dim Q^*=\dim Q = (m-r)n$. 

	Finally, we claim that every non-zero matrix in $Q^*$ has rank greater than $r$. Indeed, fix some non-zero matrix $A \in Q^*$, and fix some $M \in Q$ with $\pi(M)=A$. As $M \in Q$, we know that $\rk M > s$, hence there is a collection of more than $s$ linearly independent rows in $M$. Of these, at most $n-m$ are in the last $n-m$ rows, hence $A$ must have a collection of more than $s-(n-m)=r$ linearly independent rows, implying that $\rk A >r$, as claimed. 
\end{proof}
Given these ingredients, the proof of \cref{thm:flanders} follows immediately.
\begin{proof}[Proof of \cref{thm:flanders}]
	By \cref{thm:general reduction}, it suffices to prove the statement for finite fields $F$. We already saw that the space $R^*\leq F^{m\times n}$ of all matrices supported in the first $r$ rows has dimension $rn$ and all of its elements have rank at most $r$. By \cref{cor:roth rectangular}, we also have some $Q^* \leq F^{m\times n}$ of dimension $(m-r)n$, all of whose non-zero elements have rank greater than $r$. Since $\dim R^* + \dim Q^*=mn$, we conclude that both of these spaces must be of maximum dimension under these constraints, as desired. 
\end{proof}

For completeness, let us also see a proof of \cref{thm:roth}, due to Meshulam \cite{MR1324051}.
\begin{proof}[Proof of \cref{thm:roth}]
	Let $F=\F_q$ be a finite field, and let $k=n-s$. If $s=n$ we may take $Q=\{0\}$, so we assume henceforth that $k>0$. Consider the set $Q$ of all polynomials of the form
	\[
		f(x) = \sum_{i=0}^{k-1} a_i x^{q^i},
	\]
	where $a_i \in \F_{q^n}$, and where we view the variable $x$ as also living in the extension field $\F_{q^n}$. The first observation is that each such $f \in Q$ is an $\F_q$-linear map $\F_{q^n} \to \F_{q^n}$, since $(x+y)^q = x^q + y^q$ for every $x,y\in \F_{q^n}$. Hence we may view $Q$ as a subset of $\End_{\F_q}(\F_{q^n}) \cong \F_q^{n\times n}$. 
	The second observation is that every non-zero $f$ in $Q$ is a polynomial of degree at most $q^{k-1}$, hence has at most $q^{k-1}$ roots in $\F_{q^n}$. This means that, viewed as a linear function, its kernel has dimension at most $k-1<k$. By the rank-nullity theorem, this means that every non-zero element of $Q$ has rank greater than $n-k=s$. It also implies that the linear map sending a coefficient vector $(a_0,\dots,a_{k-1})$ to the corresponding element $\sum a_i x^{q^i}$ of $Q$ is injective, and thus that distinct coefficient vectors $(a_0,\dots,a_{k-1})$ yield distinct elements of $Q$. The final observation is that any $\F_q$-linear combination of two $f,f' \in Q$ is another element of $Q$, hence $Q$ is actually a subspace of $\F_q^{n\times n}$. We have $q^n$ choices for each of the $k$ coefficients $a_i$, hence $\ab Q = (q^n)^k$, which in turn implies that $\dim Q = kn=(n-s)n$, as desired. 
\end{proof}

Just as we proved \cref{lem:algebraically open} for all algebraically open fields, and not only for finite fields, we could also extend this proof to cover a wider class of fields. Namely, let us call a field $F$ \emph{cyclically open} if for all positive integers $n$, there is a cyclic field extension $K\supset F$ of degree $n$. Then one can check that (essentially) the same proof works to show that if $F$ is cyclically open, then there is $Q \leq F^{n\times n}$ of dimension $\dim Q = (n-s)n$, all of whose non-zero elements have rank greater than $s$. Rather than working with the polynomials $f$ as defined above, instead one works with (formal) polynomials in the variable $\sigma$, where $\sigma$ is a generator of the Galois group $\Gal(K/F)$, which is assumed to be cyclic of order $n$. Each such polynomial yields an $F$-linear map $K \to K$, and one can verify that the kernel of this map has dimension strictly less than the degree of the polynomial, which was the key property we used in the proof of \cref{thm:roth}.

As with algebraically open fields, there are numerous examples of cyclically open fields that are not finite, such as $\Q$ and all number fields \cite[Theorem X.6]{MR2467155}. In fact, we are not aware of any example of an algebraically open field that is not cyclically open (although we believe that such examples must surely exist).

\paragraph{Acknowledgments:}
I am indebted to Vivian Kuperberg, Youming Qiao, and Avi Wigderson for enlightening discussions on this topic, and to the anonymous referees for helpful comments on an earlier version of this article.

\paragraph{AI disclosure statement}
	ChatGPT 5.5 was used during the final revision of this manuscript to check for typos. All of the ideas and writing in this paper are due entirely to the author.


\end{document}